\documentclass{amsart}

\title{The set of distances in a Polish metric space}
\author{John D. Clemens}
\address{Boise State University, 1910 University Dr., Boise, ID 83725}
\email{johnclemens@boisestate,edu}
\date{}
\thanks{The main result appeared in the author's Ph.D. thesis \cite{diss}, although a simplified proof is given here.}

\usepackage{amsmath}
\usepackage{amssymb}
\usepackage{latexsym}
\usepackage{amsopn}
\usepackage{amsthm}

\def\bbN{{\mathbb N}}
\def\bbR{{\mathbb R}}

\def\bbZ{{\mathbb Z}}

\DeclareRobustCommand{\qed}{%
  \ifmmode 
  \else \leavevmode\unskip\penalty9999 \hbox{}\nobreak\hfill
  \fi
  \quad\hbox{\qedsymbol}}

\renewcommand{\qedsymbol}{\mbox{$\Box$}}

\renewenvironment{proof}[1][\proofname]{\par
  \normalfont
\trivlist 
  \item[\hskip\labelsep
  \bfseries 
    #1:]\ignorespaces
}{\qed \endtrivlist }

\newtheorem{thm}{Theorem}
\newtheorem*{thm*}{Theorem}
\newtheorem{cor}[thm]{Corollary}
\newtheorem{prop}[thm]{Proposition}
\newtheorem{lem}[thm]{Lemma}

\theoremstyle{definition}
\newtheorem*{dfn*}{Definition}
\newtheorem*{question*}{Question}

\DeclareMathOperator{\Dist}{Dist}

\def\bSi{{\boldsymbol{\Sigma}}}

\def\c{\mathcal{C}}
\def\n{\mathcal{N}}

\begin{document}

\maketitle

\begin{abstract}
We show that a set of non-negative reals is the distance set of a
separable complete metric space if and only if it is either countable or is 
an analytic set which has 0 as a limit point. We also consider spaces
with simpler distance sets.
\end{abstract}

In this article we consider the possible sets of distances in Polish metric 
spaces. 
A {\it Polish metric space} is a pair $(X,d)$, where $X$ is a Polish
space (a separable, completely-metrizable space) and $d$ is a complete, 
compatible metric for $X$.  Given a Polish metric space $(X,d)$, we can 
consider its set of distances, $\Dist(X,d)$:

\begin{dfn*} Let$(X,d)$ be a Polish metric space. The {\it distance set} of $(X,d)$, $\Dist(X,d)$, is:
\[ \Dist(X,d) =  \{ d(x,y) : x,y \in X \} . \]
\end{dfn*}
Our main result will be to characterize which sets of reals can be the distance
set of some Polish metric space. We also characterize the spaces with
some simpler distance sets. Our interest in distance sets
is partially motivated by the problem of classifying metric spaces up to
isometry, since distance sets form an isometry invariant, although
generally not a complete invariant. An overview of the isometry problem
may be found in \cite{cgk}.

In the first section we present some basic facts about distance sets, and in 
the next section we present some preliminaries about spaces with 
countable distance sets.
In Section~\ref{sec:main} we present the main result that the
distance sets of Polish metric spaces are precisely the analytic sets which
are either countable or have 0 as a limit point. Following that we draw
some corollaries about spaces with compact and $K_{\sigma}$ distance sets,
and in the final section we discuss larger point configurations.

\section{The set of distances}

Distance sets have been studied in several contexts. Much of the work has
been on the distance sets of subsets of the spaces $\bbZ$ 
or $\bbR^n$ with the usual metrics.  One of the earliest
results is Steinhaus's theorem (in \cite{steinhaus}) that the
distance set of a subset of $\bbR$ of positive measure contains a
(right-) neighborhood of 0. Sierpi{\'{n}}ski also showed in \cite{sierpinski}
that a distance set could be more complicated than the original set by
producing a $G_{\delta}$ subset of $\bbR$ whose distance set is
$\bSi^1_1$-complete.

In the case that the metric space in question is a subset of $\bbZ$,
the distance set is generally known as a {\em difference
set}. Schmerl in \cite{schmerl} has characterized how complicated the set
of difference sets is, showing that in a descriptive sense there can be no
simple characterization of when a subset of $\bbZ$ is a difference set:
\begin{thm*}[Schmerl] The set 
$\{A \subseteq \bbZ : \text{$A$ is a difference set} \}$ 
is $\bSi^1_1$-complete.
\end{thm*}

There are also several results characterizing which sets can be the set of
distances in some type of metric space. 
In the most general context, we have the following result (see
\cite{kellynordhaus}):

\begin{thm*}[Kelly and Nordhaus] Any set of non-negative reals containing 0
is the distance set of some metric space.
\end{thm*}
Several results characterize which sets can
be distance sets for subsets of $\bbR^n$ (see for instance \cite{kelly}).
More generally, in \cite{kellynordhaus} the authors characterize which sets
can be the set of distances of a separable metric space:

\begin{thm*}[Kelly and Nordhaus] A non-negative set of reals (containing 0)
is the set of distances of some separable metric space if and only if it is
either countable or has 0 as a limit point.
\end{thm*}

Here we will consider {\it Polish} metric spaces,
so we will need to produce complete metrics. 
Two properties of the distance set of a Polish metric space are clear.  
First, since $d$ is a continuous map
from the Polish space $X^{2}$ to $\bbR$, $\Dist(X,d)$ is an analytic set of
non-negative reals containing 0.
Second, if $\Dist(X,d)$  is uncountable then $X$ must also be
uncountable, so 0 must be a limit point of $\Dist(X,d)$
in order for the space to be separable.

It turns out that these two conditions are sufficient for a non-negative set
of reals to be the distance set of some Polish metric space $(X,d)$.
This will be proved in the main result of this paper,
Theorem~\ref{thm:distance}. 
In the next section and Section~\ref{sec:compact} we also consider 
Polish metric spaces which have more restrictive conditions on
their distance sets.  The distances set generally
retains some of the topological properties of the original space, and 
there is also some correspondence between the topological
complexity of distance sets and the complexity of the isometry problem for
a given class of spaces; see \cite{isom}, \cite{cgk}, or \cite{gaokechris} 
for results on these classification problems. For
brevity, we will always assume that 0 is contained in a putative set of
distances.

\section{Spaces with countable distance sets}

We begin by considering countable distance sets, which will simplify our
later results (by countable we mean either finite or countably infinite).
Recall that a metric $d$ is an {\it ultrametric} if for all $x$,
$y$, and $z$ we have $d(x,z) \leq \max(d(x,y),d(y,z))$, which is equivalent 
to saying that the longest two sides in any triangle have the same length.

\begin{lem}
\label{lem:countable}
Let $A$ be a countable set of positive reals. Then there is a discrete
Polish ultrametric space $(X,d)$ with $\Dist(X,d) = A$.
\end{lem}
\begin{proof}
Let $A \setminus \{0\} = \{ a_i : i \in \bbN \}$, where we allow repetitions 
when $A$ is finite, and let $a \neq 0$ be some fixed element of $A$. Let  
\[ X= \{x_i : i \in \bbN\} \cup \{y_i : i \in \bbN\} \]
and define $d$ by:
\begin{align*}
d(x_i,y_i) & = a_i \\
d(x_i,x_j) & = d(y_i,y_j) = d(x_i,y_j) = \max(a,a_i,a_j)
 \text{ for $i \neq j$}  .
\end{align*}
It is straightforward to check that this defines a discrete Polish metric
space since there are no non-trivial Cauchy sequences. To verify that it is
an ultrametric, it suffices to consider two representative cases.
First, let $x_i$, $x_j$, and $x_k$ be distinct, where we may assume
$a_i \leq a_j \leq a_k$. 
Then $d(x_i,x_k) = d(x_j,x_k) = \max(a,a_k) \geq \max(a,a_j) = d(x_i,x_j)$. 
Second, given $x_i$, $y_i$, and $x_j$ with $i\neq j$
we have $d(x_i,x_j) = d(y_i,x_j) = \max(a,a_i,a_j) \geq a_i = d(x_i,y_i)$.
\end{proof}

Every discrete Polish metric space is countable and hence has a countable
distance set, as does a Polish ultrametric space (since its distance set
is equal to the distance set of a countable dense subset). We thus have:

\begin{cor} A set $A \subseteq [0,\infty)$ is the distances set of some
discrete Polish metric space if and only if $A$ is the distance set of some
Polish ultrametric space if and only if $A$ is countable.
\end{cor}

We consider in more detail a common special case.

\begin{prop} 
A set $A \subseteq [0, \infty)$ is the set of distances of some perfect, compact, ultrametric space if and only if $A$ can be enumerated as a countable decreasing sequence $\langle d_i : i \geq 0 \rangle$ with $\lim_{i \rightarrow \infty} d_i = 0$.
\end{prop}
\begin{proof}
Suppose $A = \{d_i : i \geq 0\}$ with $d_i > d_{i+1}$ and $\lim_{i
\rightarrow \infty} d_i = 0$. We take as our underlying set $X = 2^{\bbN}$.
For $\alpha, \beta \in 2^{\bbN}$ with $\alpha \neq \beta$ we define
\[ d(\alpha,\beta) = d_{n(\alpha,\beta)}, \]
where $n(\alpha,\beta)$ is the least $n$ such that $\alpha(n) \neq
\beta(n)$.
This metric is equivalent to the usual metric on the Cantor space 
$2^{\bbN}$ given by $d(\alpha,\beta) = 2^{-n(\alpha,\beta)}$,
and hence the space is compact and perfect. To see that the
metric is an ultrametric, note that for $\alpha$, $\beta$ and $\gamma$ we have
$n(\alpha,\gamma) \geq \min( n(\alpha,\beta), n(\beta,\gamma))$
so $d(\alpha,\gamma) \leq \max( d(\alpha,\beta), d(\beta,\gamma))$.

For the other direction, let $(X,d)$ be a perfect, compact, ultrametric
space. Since the space is perfect it must have distances arbitrarily close
to 0. It is then sufficient to observe that in a compact ultrametric space
$(X,d)$, for any $b >0$ the set 
\[ \{ d(x,y) : \text{$x,y \in X$ and $d(x,y) \geq b$} \} \]
is finite, since we can then take $d_0$ to be the largest distance, $d_1$ to be the next largest, and so forth.
\end{proof}

\section{The main result}
\label{sec:main}

\begin{thm}
\label{thm:distance}
A set $A \subseteq [0, \infty)$ is the distance set of some Polish metric 
space if and only if $A$ is either countable or is an analytic set with 
$0$ as a limit point.  
\end{thm}
\begin{proof}
The conditions are necessary, as noted earlier. The case of $A$ countable was
handled by Lemma~\ref{lem:countable}, so let
$A$ be analytic with $0$ as a limit point.
We will first assume that $A \subseteq [0,1)$, 
and handle the general case at the end. We may identify sequences from 
$\{0,1\}$ with reals in $[0,1]$ via the map
\[ \pi: \alpha \mapsto \sum_{i \in \bbN} \frac{\alpha(i)}{2^{i+1}} . \]
This is a continuous surjection from the Cantor space $\c = 2^{\bbN}$ onto
$[0,1]$, so the inverse image of an analytic set of reals is an
analytic subset of the Cantor space.  The map is one-to-one 
except for those points with eventually constant binary representations,
where it is two-to-one.
We may thus represent $A$ by an analytic subset of the Cantor space, which
will be the projection of a closed subset of $\c \times \n$
(where $\n$ is the Baire space $\bbN^{\bbN}$), and hence
the projection of the infinite branches $[T]$ of a pruned tree $T$ on 
$2 \times \bbN$.
We fix a pruned tree $T$ on $2 \times \bbN$ such that
\[ a \in A \Longleftrightarrow (\exists \alpha) (\exists \beta) 
\ [\forall n \ (\alpha\upharpoonright n, \beta \upharpoonright n ) \in T \text{ and }
\pi(\alpha)=a] . \]
For technical reasons, we choose $T$ so that there is a unique branch
projecting to 0 (i.e., to the infinite sequence $0^{\infty}$); 
since $A \setminus \{0\}$ is still an analytic set we can start
with a tree projecting to $A \setminus \{0\}$ and add the branch
$(0^{\infty},0^{\infty})$. We now let
\[ T^{*} = \{ s \in 2^{< \bbN} : (\exists b \in \bbN^{< \bbN})\ [(s,b)
\in T] \} . \]
Then $T^{*}$ is a pruned tree on 2 with $[T^{*}] = \overline{\pi^{-1}[A]}$, 
the closure of $\pi^{-1}[A]$ in $2^{\bbN}$. 
For each $s \in T^{*}$, we can thus pick $\alpha_s \in \pi^{-1}[A]$ with 
$s \sqsubset \alpha_s \neq 0^{\infty}$,
and let $d_s = \pi(\alpha_s) \in A \setminus \{0\}$.
Choose finally a decreasing sequence $\langle \epsilon_i \rangle_{i \in 
\bbN}$ from among the $d_s$'s with  
$\epsilon_i \in A$ and $\epsilon_i < \frac{1}{2^i}$, which
is possible since 0 is a limit point of $A$ and the $d_s$'s are dense in
$A$.

We set $X = [T]$, and define $d$ as follows. For $x_1=(\alpha_1, \beta_1)$ and
$x_2=(\alpha_2, \beta_2)$ with $x_1 \neq x_2$ we let $(s,b)$ be the maximal 
mutual predecessor of $x_1$ and $x_2$ in $T$. We then let:
\[ d(x_1,x_2) = \begin{cases} \max (\pi(\alpha_1), \pi(\alpha_2)) & 
\text{if $s = 0^k$ for some $k$} \\
\text{least $\epsilon_i$ s.t. $2^{-|s|} \leq \epsilon_i \leq d_s$} & 
\text{if $s \neq 0^k$ and such $\epsilon_i$ exists} \\
d_s & \text{otherwise} \end{cases} \]

We check that this is a metric, which amounts to verifying the triangle
inequality. Fix $x_i = (\alpha_i, \beta_i)$ distinct for $i = 1,2,3$.
\begin{enumerate}
\item If $x_1$, $x_2$ and $x_3$ have a mutual maximal predecessor $(s,b)$,
then:
\begin{enumerate} 
\item If $s= 0^k$ then each distance $d(x_i,x_j)=
\max(\pi(\alpha_i),\pi(\alpha_j))$,
so this is an ultrametric triangle (the longest two distances are equal).
\item If $s \neq 0^k$ then all the distances are the same, since they
depend only on $s$. 
\end{enumerate}
\item Otherwise, two of the branches agree longer than they do with the
third. We may assume that we have $(t,c) \sqsubset (s,b)$ with $(s,b)$ the
maximal mutual predecessor of $x_1$ and $x_2$, and $(t,c)$ the maximal
predecessor of $x_3$ and $(s,b)$.  We have three sub-cases:
\begin{enumerate}
\item If $s = 0^k$ then again this is an ultrametric triangle.
\item If $s \neq 0^k$ and $t = 0^j$ then 
\begin{align*}
d(x_1,x_3) & = \max(\pi(\alpha_1), \pi(\alpha_3)) \\
d(x_2,x_3) & = \max (\pi(\alpha_2), \pi(\alpha_3)) \\
2^{-|s|} \leq d(x_1,x_2) & \leq d_s \leq \pi(s) + 2^{-|s|},
\end{align*}
where we identify $s$ with the sequence $s$ followed by 
all 0's.  Note that since $s \neq 0^{|s|}$ we have $2^{-|s|} \leq \pi(s)$
and so
\[ | \pi(\alpha_1) - \pi(\alpha_2) | \leq 2^{-|s|} \leq d(x_1,x_2) 
\leq \pi(s) + 2^{-|s|} \leq \pi(\alpha_1) + \pi(\alpha_2) . \]
We then see that in all cases we have $| d(x_1,x_3)-d(x_2,x_3)| \leq
|\pi(\alpha_1) - \pi(\alpha_2)|$, so that we have
\[ | d(x_1,x_3)- d(x_2,x_3) | \leq d(x_1,x_2) \leq d(x_1,x_3)+d(x_2,x_3) , \]
which guarantees the triangle inequality.
\item If $t \neq 0^j$ then $d(x_1,x_3) = d(x_2,x_3)$ since these
distances depend only on $t$. 
If this distance is some $\epsilon_i$ then we have $d(x_1,x_2)$ less than
or equal to this distance, since it will either be this $\epsilon_i$ or some
smaller $\epsilon_j$. If this distance is $d_t$, then we have 
$d(x_1,x_2) \leq d_s \leq d_t + 2^{-|t|} \leq 2 \cdot d_t$ since 
$t \neq 0^j$, so in all cases the triangle inequality holds here.
\end{enumerate}
\end{enumerate}

Thus $d$ is a metric. To check that it is complete, let $\langle x_i
\rangle$ be a $d$-Cauchy sequence with $x_i = (\alpha_i,\beta_i) \in [T]$.
Since $2^{\bbN}$ is compact, there is a subsequence $\langle a_{i_n}
\rangle_{n \in \bbN}$ converging to some $a_{\infty}$ in the usual topology of
$2^{\bbN}$. If $a_{\infty} = 0^{\infty}$ then $\langle x_{i_n} \rangle$ (and
hence $\langle x_i \rangle$) converges to the point
$(0^{\infty},0^{\infty})$ in $X$.
If $a_{\infty} \neq 0^{\infty}$, then both the sets $\{\pi(\alpha_{i_n}) :
n \in \bbN\}$ and $\{ d_s : \text{$s \sqsubset \alpha_{i_n}$ for some $n
\in \bbN$}\}$ are bounded away from 0, so for sufficiently small $\epsilon$
and large enough $n$ and $m$ so that $d(x_{i_n},x_{i_m}) < \epsilon$ we
must have (per the second clause of the definition of $d$) that
$d(x_{i_n},x_{i_m})$ is equal to some $\epsilon_i$ with $2^{-|s|} \leq
\epsilon_i < \epsilon$, where $(s,b)$ is the maximal mutual predecessor
of $x_{i_n}$ and $x_{i_m}$ in $T$. Hence the mutual predecessors must have large
length, so that the sequence $\langle x_{i_n} \rangle$
converges in the usual topology of $2^{\bbN} \times \bbN^{\bbN}$. Since $[T]$ is
closed in $2^{\bbN} \times \bbN^{\bbN}$, $\langle x_{i_n} \rangle$
converges to some $(\alpha_{\infty},\beta_{\infty}) \in [T]=X$ in this
topology, and must also converge according to $d$, so the sequence $\langle
x_i \rangle$ converges to this point as well.

To see that $(X,d)$ is separable, we pick for each $(s,b) \in T$ some 
branch $(\alpha, \beta)$ in $[T]$ extending $(s,b)$. 
This yields a countable set $D$, and for any branch $(\alpha', \beta') \in
[T]$ if we choose the $(\alpha,\beta) \in D$ corresponding to $(\alpha'
\upharpoonright n, \beta' \upharpoonright n)$ then
$d((\alpha',\beta'),(\alpha,\beta)) < 2^{-n +1}$.
Finally, it is clear that $\Dist(X,d) = A$, since all of the distances defined
are in $A$, and distances from the branch with $\alpha = 0^{\infty}$ will
include all elements of $A$. This completes the proof when $A \subseteq
[0,1)$.

For the general case of $A \subseteq [0, \infty )$, let $A_n = A \cap
[0,n)$.  Since the $A_n$'s satisfy the hypotheses of the theorem, 
we can construct
$(X_n,d_n)$ as above (stretching the metric by $n$) such that 
$\Dist(X_n,d_n) = A_n$.  Now choose a sequence $\langle \delta_n
\rangle$ with $\delta_n \in A_n$, $\delta_n \geq \frac{1}{2}\sup A_n$, and
$\delta_n \leq \delta_{n+1}$.  We let $X$ be the disjoint union
$\bigsqcup_{n \geq 1} X_n$.  We set $d = d_n$ on each $X_n$, and for $x
\in X_n$, $y \in X_m$ with $n<m$ we let $d(x,y) = \delta_m$.  The
conditions on the $\delta_n$'s guarantee that this is a complete 
metric and adds no
distances other than those in the $A_n$'s, so that $\Dist(X,d) = \bigcup_{n
\geq 1} A_n = A$.
\end{proof}

There is a lack of uniformity in the above construction.  For a given
set $A$, we used a tree representation of $A$ in order to
construct our space, and different trees may give rise to non-isometric
spaces. A subsequent article (\cite{dist2}) we will consider how close
the distance set is to being a complete invariant for isometry, and show
that it is very far from being complete. In fact, for many analytic sets
$A$ the classification of spaces with distance set $A$ up to isometry is as
complicated as the classification of all Polish metric spaces.

The main theorem produces the following corollary. Recall that a 
{\it zero-dimensional} space is one in which there is a basis consisting
of clopen sets.  Although zero-dimensional spaces are a special class of
Polish metric spaces, their distance sets
can be as complicated as those of arbitrary Polish metric spaces:

\begin{cor} A set $A \subseteq [0,\infty)$ is the set of distances of some
zero-dimens\-ional Polish metric space if and only if either $A$ is countable
or $A$ is analytic with 0 as a limit point.
\end{cor}
\begin{proof}
We check that the construction in 
Theorem~\ref{thm:distance} in fact produces zero-dimensional spaces. For the
countable case, this is immediate since ultrametric spaces are
zero-dimensional. For
the uncountable case, every point other than the branch corresponding to 0
has a clopen basis, since the local topology there is the subspace
topology of $\c \times \n$. Around the point 0, if we fix some sufficiently 
large $n_0$ and consider the set
\begin{align*}
G & = \{  (\alpha,\beta) : \alpha < 2^{-n_0} \} \ \cup\ \{ (\alpha,\beta): 
\alpha \upharpoonright n_0 = 0^{n_0-1}\smallfrown 1 \} \ \cup \\
& \qquad  \{ (\alpha,\beta): \alpha \upharpoonright (n_0+1) = 0^{n_0}
\smallfrown 1 \} \\
& = \{  (\alpha,\beta) : \alpha \leq 2^{-n_0} \} \ \cup\ \{ (\alpha,\beta): 
\alpha \upharpoonright n_0 = 0^{n_0-1}\smallfrown 1 \} \ \cup \\
& \qquad \{ (\alpha,\beta): \alpha \upharpoonright (n_0+1) = 0^{n_0}
\smallfrown 1 \},
\end{align*}
then $G$ is clopen since the first expression is open and the second is 
closed (we need both of the last two sets in each expression to account for
both the eventually 0 and eventually 1 representation of $2^{-n_0}$), and $0
\in G \subseteq B_{2^{-(n_0-1)}}(0)$ so that there is a clopen basis at 0
as well.
\end{proof}

\section{Spaces with compact and $K_{\sigma}$ distance sets}
\label{sec:compact}

The proof of Theorem~\ref{thm:distance} allow us to characterize the
Polish metric spaces with compact distance sets:

\begin{prop}
\label{prop:compact}
A set $A \subseteq [0,\infty)$ is the set of distances of some
compact Polish metric space if and only if either $A$ is finite or $A$ is 
compact and 0 is a limit point of $A$.
\end{prop}
\begin{proof}
The conditions are necessary since the distance set is the continuous image
of the compact set $X^2$ and compact discrete spaces are finite.
The case of $A$ finite is handled as in Lemma~\ref{lem:countable},
including only finitely many points, so let $A$ be compact with 0 as a
limit point. By scaling we may assume $A \subseteq [0,1)$, so 
following the proof of the main theorem, $\pi^{-1}[A]$ will be a compact
subset of $2^{\bbN}$. We may take a tree $T^{\ast}$ on $2$ such that
$[T^{\ast}] = \pi^{-1}[A]$ and let $T= \{ (s,0^{|s|}) : s \in T^{\ast}\}$
(so $[T] = [T^{\ast}] \times \{0^{\infty}\}$). 
We construct $(X,d)$ as before so that $\Dist(X,d) = A$.

We check that $(X,d)$ so constructed is totally bounded. Fix 
$\epsilon >0$, and let $n_0$ be such that there is some $\epsilon_i$ with 
$2^{-n_0} < \epsilon_i < \epsilon$.
Set:
\[ D_{\epsilon} = \{ (\alpha_s,0^{\infty}) : s \in T \text{ and } |s| \leq
n_0 \} .\]
Then $D_{\epsilon}$ is finite, and every point in $X$
is within distance $\epsilon$ of some point in $D_{\epsilon}$. This is
clear from our metric since any point must agree with one of these
branches on its first $n_0$ coordinates; if these are not all 0, 
we have an $\epsilon_i$ with $2^{-n_0} < \epsilon_i < \epsilon$, and if
they are all 0 then the distance to $(\alpha_{0^{n_0}},0^{\infty})$ is at most $2^{-n_0}$.
\end{proof}

Since the map $\pi$ is at most two-to-one, we have that the space $X$ will
be countable when $A$ is countable, and hence we also have:

\begin{prop} A set $A \subseteq [0,\infty)$ is the distance set of some 
countable compact Polish metric space if and only if $A$ is a countable
compact set which is either finite or has 0 as a limit point.
\label{thm:cntblecompact}
\end{prop}

Finally, we can characterize the spaces with $K_{\sigma}$ distance sets:

\begin{prop} The following are equivalent:
\begin{enumerate}
\item $A \subseteq [0,\infty)$ is the distance set of a
locally compact Polish metric space.
\item $A$ is the distance set of a $\sigma$-compact Polish metric space.
\item $A$ is either countable or $A$ is $K_{\sigma}$ with 0 as a limit
point.
\end{enumerate}
\end{prop}
\begin{proof}
(1) $\Rightarrow$ (2) follows from the fact that any locally compact
Polish space is $\sigma$-compact,
and (2) $\Rightarrow$ (3) is straightforward.
The case of (3) $\Rightarrow$ (1) follows from
Proposition~\ref{prop:compact} in the same way that the general case of
Theorem~\ref{thm:distance} follows from the case where $A$ is bounded.
\end{proof}

\section{Larger point configurations}
\label{sec:spec}

As shown in \cite{dist2}, the distance set of a Polish metric space is very
far from being a complete invariant for isometry. We can try to improve on
this by considering configurations of a larger number of points.
\begin{dfn*} For $(X,d)$ a Polish metric space and $n \geq 2$, let 
the {\it $n$-point spectrum}, $\text{Spec}_n(X,d)$, be the set:
\[ \{ \langle d_{i,j} \rangle_{i<j<n} : (\exists
x_0,\ldots,x_{n-1}\in X) (\forall i<j<n)\ [ d_{i,j} = d(x_i,x_j) ] \}. \]
\end{dfn*}
Then $\text{Spec}_2(X,d) = \Dist(X,d)$. In general, 
$\text{Spec}_n(X,d)$ is an analytic subset of $\bbR^{\frac{n(n-1)}{2}}$
and completely determines $\text{Spec}_m(X,d)$ for $m < n$.
We can then ask:

\begin{question*} For $n \geq 3$, what sets can be $\text{Spec}_n(X,d)$ for
some Polish metric space $(X,d)$?
\end{question*}

The characterization of the possible $n$-point spectra does not
seem as simple as in the case $n=2$. Consider the case
of $n=3$. Each element of $\text{Spec}_3(X,d)$ must be a triple satisfying the
triangle inequality, but after restricting to the closed subset of such
metric triples the issue seems to be primarily combinatorial.  
For instance, fixing a finite set $\mathcal{T}=\{T_1,\ldots,T_n\}$ of 
metric triples, the question of whether $\mathcal{T}$ is
$\text{Spec}_3(X,d)$ for some $k$-element space $(X,d)$ is equivalent to
the existence of a coloring of the the complete graph
$\mathcal{K}_k$ on $k$ vertices (using as colors the distances in elements
of $\mathcal{T}$) such that the induced colorings on sub-triangles are
precisely $\mathcal{T}$.

Although Polish metric spaces are not generally characterized up to isometry 
by their distances sets, or even by the sequence $\langle \text{Spec}_n(X,d)
: n \in \bbN\rangle$, there are two cases in which this is true.  One
is the case of compact metric spaces.
Two compact metric spaces are isometric if and only if they have the same
$n$-point spectra for all $n\geq 2$ (Theorem $3.27\frac{1}{2}$ of
\cite{gromov}).  Here the spectra are compact subsets of
$\bbR^{\frac{n(n-1)}{2}}$ so this sequence of compact sets is
a complete invariant for isometry and shows that the isometry
relation on compact metric spaces is concretely classifiable.

A second case is that of {\it
ultra-homo\-geneous} spaces, those in which any isometry between finite
subsets of the space extends to an isometry of the whole space. Here again
two ultra-homogeneous Polish metric spaces with equal spectra for all $n$
are isometric. These spectra are no longer compact, and so do not provide a 
concrete classification up to isometry.
It would be interesting to know the possible spectra in this case as an
indication of the complexity of the isometry relation on ultra-homogeneous
Polish metric spaces.

\end{document}